\documentclass[12pt]{article}
\usepackage{epsfig}
\usepackage{amssymb}
\usepackage{amsmath}
\usepackage{graphicx}
\usepackage{color}

\newcommand\PG{{\rm PG}}
\newcommand\PGL{{\mbox{PGL}}}

\newcommand\GF{{\rm\mbox{GF}}}

\renewcommand{\P}{\mathcal{P}}

\newcommand{\Label}{\label}

\newtheorem{theorem}{Theorem}[section]
\newtheorem{lemma}[theorem]{Lemma}
\newtheorem{corollary}[theorem]{Corollary}

\newenvironment{proof}{\noindent{\bf Proof}\hspace{0.5em}}
    { \null  \hfill $\square$ \par}

\usepackage[mathscr]{eucal}

\newcommand{\bbb}{{\mathscr B}}

\newcommand{\bbc}{{\mathcal C}}

\newcommand{\bbn}{{\mathcal N}}

\renewcommand{\r}{{q}}

\renewcommand{\S}{{\cal S}}
\newcommand{\C}{{\cal C}}
\newcommand{\N}{{\cal N}}

\newcommand{\si}{\Sigma_\infty}
\newcommand{\li}{\ell_\infty}
\newcommand{\abb}{{\cal A(\cal S)}}
\newcommand{\pbb}{{\cal P(\cal S)}}

\newcommand{\orsps}{order-$\r$-subplanes}
\newcommand{\orsls}{order-$\r$-sublines}
\newcommand{\orsp}{order-$\r$-subplane}
\newcommand{\orsl}{order-$\r$-subline}

\newcommand{\takeaway}{\backslash}

\begin{document}

\title{A Characterisation of Tangent Subplanes of $\PG(2,\r^3)$}
\date{}

\author{S.G. Barwick and Wen-Ai Jackson
\date{\today}
\\ School of Mathematics, University of Adelaide\\
Adelaide 5005, Australia
\\ \\
}

\maketitle

Corresponding Author: Dr Susan Barwick, University of Adelaide, Adelaide
5005, Australia. Phone: +61 8 8303 3983, Fax: +61 8 8303 3696, email:
susan.barwick@adelaide.edu.au



\begin{abstract}
In \cite{barw12}, the authors determine the representation of \orsp s and
\orsl s of $\PG(2,q^3)$ in the Bruck-Bose representation in $\PG(6,q)$. 
In particular, they showed that an \orsp\ of $\PG(2,q^3)$ corresponds to a
certain ruled surface in $\PG(6,q)$. In this article we show that the
converse holds,
namely that any ruled surface satisfying the required properties 
corresponds to a tangent \orsp\ of $\PG(2,q^3)$. 
\end{abstract}


\section{Introduction}\Label{section1}

We begin with a brief introduction to  $2$-spreads in $\PG(5,q)$, and the
Bruck-Bose representation of $\PG(2,q^3)$ in $\PG(6,q)$, and introduce the
notation we will use. 

A 2-{\em spread} of $\PG(5,\r)$ is a set of $\r^3+1$ planes that partition
$\PG(5,\r)$. 
The following construction of a regular 2-spread of $\PG(5,\r)$ will be
needed. Embed $\PG(5,\r)$ in $\PG(5,\r^3)$ and let $g$ be a line of
$\PG(5,\r^3)$ disjoint from $\PG(5,\r)$. The Frobenius automorphism of
$\GF(q^3)$ where
$x\mapsto x^q$ induces a collineation of $\PG(5,q^3)$. Let $g^\r$, $g^{\r^2}$ be the
conjugate lines of $g$; both of these are disjoint from $\PG(5,\r)$. Let $P_i$ be
a point on $g$; then the plane $\langle P_i,P_i^\r,P_i^{\r^2}\rangle$ meets
$\PG(5,\r)$ in a plane. As $P_i$ ranges over all the points of  $g$, we get
$\r^3+1$ planes of $\PG(5,\r)$ that partition $\PG(5,q)$. These planes form a
regular spread $\S$ of $\PG(5,\r)$. The lines $g$, $g^\r$, $g^{\r^2}$ are called the (conjugate
skew) {\em transversal lines} of the spread $\S$. Conversely, given a regular 2-spread
in $\PG(5,\r)$,
there is a unique set of three (conjugate skew) transversal lines in $\PG(5,\r^3)$ that generate
$\S$ in this way. 
See \cite{hirs91} for
more information on 2-spreads.

We work in linear representation of a finite
translation plane $\P$ of dimension at most three over its kernel,
an idea which was developed independently by
Andr\'{e}~\cite{andr54} and Bruck and Bose
\cite{bruc64,bruc66}. 
Let $\si$ be a hyperplane of $\PG(6,\r)$ and let $\S$ be a 2-spread
of $\si$. We use the phrase {\em a subspace of $\PG(6,\r)\takeaway\si$} to
  mean a subspace of $\PG(6,\r)$ that is not contained in $\si$.  Consider the following incidence
structure:
the \emph{points} of $\abb$ are the points of $\PG(6,\r)\takeaway\si$; the \emph{lines} of $\abb$ are the 3-spaces of $\PG(6,\r)\takeaway\si$ that contain
  an element of $\S$; and \emph{incidence} in $\abb$ is induced by incidence in
  $\PG(6,\r)$.
Then the incidence structure $\abb$ is an affine plane of order $\r^3$. We
can complete $\abb$ to a projective plane $\pbb$; the points on the line at
infinity $\li$ have a natural correspondence to the elements of the 2-spread $\S$.
The projective plane $\pbb$ is the Desarguesian plane $\PG(2,\r^3)$ if and
only if $\S$ is a regular 2-spread of $\si\cong\PG(5,\r)$ (see \cite{bruc69}).

We will be using the cubic extension $\PG(6,q^3)$ of $\PG(6,q)$. If $K$ is
a subspace or curve of $\PG(6,q)$, we use $K^*$ to denote the natural
extension of $K$ to $\PG(6,q^3)$. 

\section{The characterisation}

In \cite{barw12}, the authors prove the following result that an \orsp\ of $\PG(2,q^3)$ corresponds to a
certain ruled surface in $\PG(6,q)$. In this article we show that the
converse holds,
namely that any ruled surface satisfying the required properties 
corresponds to a tangent \orsp\ of $\PG(2,q^3)$. We use the notation of
Section~\ref{section1} and recall the following
result.

\begin{theorem}{\rm \cite[\rm Theorem 2.7]{barw12}}\Label{subplane-tangent}
  Let $B$ be an \orsp\ of $\PG(2,\r^3)$ that is tangent to $\li$ in the
  point $T$. Let $\pi_T$ be the spread element corresponding to $T$.
Then $B$ determines a set $\bbb$  of points in $\PG(6,\r)$ (where the
affine points of $B$ correspond to the affine points of $\bbb$) such that:
\begin{enumerate}
\item[{\rm (a)}]  $\bbb$ is a ruled surface with conic directrix $\bbc$ contained in
  the plane $\pi_T\in\S$, and normal rational curve directrix $\bbn$ contained in a
  3-space $\Sigma$ that meets $\si$ in a spread element (distinct from
  $\pi_T$). The points of $\bbb$ lie on $\r+1$ pairwise disjoint generator lines joining $\bbc$ to $\bbn$.
\item[{\rm (b)}]  The $\r+1$ generator lines of $\bbb$ joining $\bbc$
  to $\bbn$ are determined by a projectivity from $\bbc$ to $\bbn$.
\item[{\rm (c)}]  When we extend $\bbb$
  to $\PG(6,\r^3)$, it contains the conjugate transversal
  lines $g,g^\r,g^{\r^2}$ of the spread $\S$.
\end{enumerate}
\end{theorem}

In this article we prove the converse of this result.

\begin{theorem}\Label{hope1}
In $\PG(6,q)$, let $\C$ be a conic in a spread element $\pi$ such that in the cubic
extension $\PG(6,q^3)$, $\C^*$ contains the three transversal points $P=\pi^*\cap
g,P^q=\pi^*\cap g^q,
P^{q^2}=\pi^*\cap g^{q^2}$. Let $\Sigma$ be a 3-space of
$\PG(6,q)\setminus\si$ about a spread element $\alpha$ distinct from
$\pi$. Let $\N$ be a normal rational curve in $\Sigma$ that in the cubic extension contains the points $Q=\alpha^*\cap
g,Q^q=\alpha^*\cap g^q,
Q^{q^2}=\alpha^*\cap g^{q^2}$. In $\PG(6,q^3)$, let $\bbb^*$ be the unique ruled surface  with
directrices $\C^*,\N^*$ defined by the projectivity that maps $P^{q^i}\mapsto
Q^{q^i}$, $i=1,2,3$. Then the ruled surface $\bbb$ of $\PG(6,q)$ corresponds to an \orsp\ of $\PG(2,q^3)$
that is tangent to $\li$. 
\end{theorem}

To simplify the following statements, we define a {\bf special conic} of a spread
element $\pi$ to be a conic that in the cubic extension $\PG(6,q^3)$ contains the
transversal points $P=\pi^*\cap g,P^q=\pi^*\cap g^q, P^{q^2}=\pi^*\cap g^{q^2}$.
A {\bf special normal rational curve} in a 3-space $\Sigma$ of
$\PG(6,q)\setminus\si$ through a spread element $\alpha\neq\pi$ is one
which in the cubic extension $\PG(6,q^3)$
contains the three transversal points
$Q=\alpha^*\cap g,Q^q=\alpha^*\cap g^q, Q^{q^2}=\alpha^*\cap g^{q^2}$. Note
that a special normal rational curve is disjoint from $\si$.

This allows us to make a compact statement that combines
Theorems~\ref{subplane-tangent} and \ref{hope1}.

\begin{corollary}
Let $\bbb$ be a ruled surface of $\PG(6,q)$ defined by a projectivity from a conic
directrix 
$\C$ to a normal rational curve directrix $\N$. 
Then $\bbb$ corresponds to an \orsp\ of $\PG(2,q^3)$ if and only if
$\C$ is a special conic in
a spread element $\pi$, $\N$ is a special
normal rational curve in a 3-space about a spread element distinct from
$\pi$, and
in the cubic extension $\PG(6,\r^3)$ of $\PG(6,\r)$, $\bbb$ contains the transversals of the regular
spread $\S$. 
\end{corollary}

We will prove this result by counting. By Theorem~\ref{subplane-tangent}, a
tangent \orsp\  of
$\PG(2,q^3)$ corresponds in $\PG(6,q)$ to a ruled surface with a special
conic directrix and a special normal
rational curve directrix that when extended to $\PG(6,q^3)$ contains the
transversals of the spread $\S$. We show the converse is true by counting
the 
number of tangent \orsp s of $\PG(2,q^3)$, and the number of such ruled
surfaces in $\PG(6,q)$ and
showing that the two sets have the same number of elements. 
We proceed with a series of lemmas.

\begin{lemma}\Label{count-orsp}
 The number of tangent \orsps\ of $\PG(2,q^3)$ through a fixed point $T$ of
  $\li$ is $q^7(q^3-1)(q^2+q+1)$.
\end{lemma}

\begin{proof} We first count the total 
 number of \orsp s in $\PG(2,\r^3)$, it is
\[
\frac{(\r^6+\r^3+1)(\r^6+\r^3)\r^6(\r^6-2\r^3+1)}{(\r^2+\r+1)(\r^2+\r)\r^2(\r^2-2\r+1)}
=\r^6(\r^6+\r^3+1)(\r^2-\r+1)(\r^2+\r+1).
\]
Next we count the number $x$ of \orsp s tangent to $\li$.  We
count in two ways the number of pairs $(m,\pi)$ where $m$ is a line of $\PG(2,q^3)$ tangent to an \orsp\ $\pi$.  We have 
\[
\r^6(\r^6+\r^3+1)(\r^2-\r+1)(\r^2+\r+1)\times (\r^2+\r+1)(\r^3-\r)=(\r^6+\r^3+1)x
\] and so $x=\r^7(\r^2-\r+1)(\r^2+\r+1)^2(\r-1)(\r+1)$. As the subgroup
$\PGL(3,q^3)$ fixing the line $\li$ is transitive on the points of $\li$,
the number of \orsp s tangent to $\li$ at the point $T\in\li$ is $x/(q^3+1)=\r^7(q^3-1)(\r^2+\r+1)$.
\end{proof}

\begin{lemma}\Label{count-nrc}
Let $\Sigma$ be a 3-space of $\PG(6,q)\setminus\si$ about a spread element. The number
  of special normal rational curves in $\Sigma$
  is $q^3(q^3-1)$.
\end{lemma}

\begin{proof}
By \cite[Theorem 2.5]{barw12}, the number of special normal rational curves
in $\Sigma$ is equal to the number of \orsls\ of a line
$\ell$ ($\ell\neq\li$) that are disjoint from $\li$. 
There are $${{q^3+1}\choose{3}}\Big/{{q+1}\choose{3}}=q^2(q^2+q+1)(q^2-q+1)$$
sublines of $\ell$. Of these,
$${{q^3}\choose{2}}\Big/{{q}\choose{2}}=q^2(q^2+q+1)$$ contain the point
$\ell\cap\li$. Hence there are $q^2(q^2+q+1)(q^2-q)=q^3(q^3-1)$ \orsl s of
$\ell$ that are disjoint from $\li$. 
\end{proof}

\begin{lemma}\Label{twopointsspecialconic}
 Two points in a spread element $\pi$ lie in a unique special
  conic of $\pi$. Further, every special conic of $\pi$ is non-degenerate.
\end{lemma}

\begin{proof}
In the cubic extension $\PG(6,q^3)$,  $\pi^*$ contain the three
transversal points 
$P=\pi^*\cap g,P^q=\pi^*\cap g^q, P^{q^2}=\pi^*\cap g^{q^2}$.
Let $A,B$ be two points of $\pi$. 
We first show that $A,B,P,P^q,P^{q^2}$ are five
points, no three
collinear.

If the line
$PP^q$ meets $\pi$ in a point $X$, then $X^q\in (PP^q)^q=P^qP^{q^2}$. As
$X\in\pi$, $X^q=X$, and so $X,P,P^q,P^{q^2}$ are collinear, a
contradiction as $P,P^q,P^{q^2}$ generate a plane and so are not collinear. 
So the lines $PP^q$, $PP^{q^2}$, $P^qP^{q^2}$ are all disjoint
from $\pi$, that is, no point of $\pi$ is on one of these lines. 

Next we show that the line $m=AB$
does not contain any of $P,P^q,P^{q^2}$. 
As $m$ is a line of $\pi$, we have $m^q=m$. If $P\in m$, then $P^q\in
m^q=m$,
 and similarly $P^{q^2}\in m$, a
contradiction. So $m$ does not contain $P$, $P^q$ or $P^{q^2}$.

Hence we can pick any two points $A,B$ of $\pi$ and  the five points
$A,B,P,P^q,P^{q^2}$ are no three collinear, and so lie in a unique non-degenerate
        conic $\C^*$ of $\pi^*$.
This
conic is fixed by the Frobenius automorphism $x\mapsto x^q$, and so $\C$ is
a conic of $\pi$. 
Note that this also means that any special conic of $\pi$ is non-degenerate.
\end{proof}

\begin{lemma}\Label{count-conics} The number of special conics in a spread
  element $\pi$ is 
  $q^2+q+1$.
\end{lemma}

\begin{proof} 
We want to count the number
of conics of $\pi$ that in the cubic extension $\pi^*$ contain the three
transversal points 
$P=\pi^*\cap g,P^q=\pi^*\cap g^q, P^{q^2}=\pi^*\cap g^{q^2}$.
By Lemma~\ref{twopointsspecialconic}, two points $A,B$ of $\pi$ lie in a
unique special conic of $\pi$.
 The number of ways to
choose $A,B$, so that the conic is distinct is
$(q^2+q+1)(q^2+q)/(q+1)q=q^2+q+1$ as required.
\end{proof}

\begin{lemma}\Label{count-triples}
 The number of triples $(\C,\N,\bbb)$ where $\C$ is a special conic
  in a fixed spread element $\pi$, $\N$ is a special normal rational curve in
  any 3-space of $\PG(6,q)\setminus\si$ about a spread element
  $\alpha\neq\pi$, and $\bbb$ is the unique ruled surface with directrices
  $\C,\N$ such that in the cubic extension $\PG(6,q^3)$, $\bbb^*$ contains the transversal lines $g,g^q,g^{q^2}$  is 
$q^9(q^3-1)(q^2+q+1)$.
\end{lemma}

\begin{proof}
In   Lemma~\ref{count-conics}
 we show that the  number of special conics in a fixed spread element
 $\pi$ is $q^2+q+1$. There are $q^3$ choices for the spread element
 $\alpha$, and each spread element lies in $q^3$ 3-spaces of
 $\PG(6,q)\setminus\si$. In Lemma~\ref{count-nrc} we showed that the number
 of special normal rational curves in a 3-space is $q^3(q^3-1)$. 
 Finally, as a projectivity is uniquely determined
  by the image of three points, in the cubic extension $\PG(6,q^3)$ there
  is a unique ruled surface $\bbb^*$ with
  directrices $\C^*$ and $\N^*$ that contains the transversal lines
  $g,g^q,g^{q^2}$ of the spread $\S$. We now show that $\bbb^*$ meets $\PG(6,\r)$ in a ruled surface $\bbb$.

  The Frobenius automorphism $\sigma\colon x\mapsto x^\r$ fixes $\C$ and $\N$
  pointwise, and also fixes the set $\{g,g^q,g^{q^2}\}$. As $q\geq 2$, $\C$
  has at least 3 points in $\pi$, and so $\C^*,\C^{*q}$ have at least six common
  points, hence $\sigma$ fixes
  $\C^*$. Similarly, $\sigma$ fixes $\N^*$.
 Thus $\sigma$ fixes $\bbb^*$ since $\bbb^*$ is determined by a
 projectivity, and the three lines
  $\{g,g^q,g^{q^2}\}$ uniquely determine this projectivity.  
As $\sigma$ fixes exactly the points of $\PG(6,q)$, it follows that
$\bbb^*$ meets $\PG(6,q)$ in a ruled surface $\bbb$ with directrices $\C,\N$. 
That is, $\bbb$ satisfies the conditions of the lemma.
Hence the
  number of triples is $(q^2+q+1)\times(q^3\times q^3\times
  q^3(q^3-1))\times1$ as required.
\end{proof}

{\bf Proof of Theorem~\ref{hope1}} To complete the proof of
Theorem~\ref{hope1}, we count the number of triples $(\C,\N,\bbb)$ where
$\bbb$ is a ruled surface of $\PG(6,q)$ that corresponds to a tangent
\orsp\ of $\PG(2,q^3)$ through a fixed point
$T$ of $\li$, and $\bbb$ has conic directrix $\C$ and normal rational curve directrix
$\N$. 
In Lemma~\ref{count-orsp} we showed that the number of tangent \orsp s
through a fixed point $T$ is $q^7(q^3-1)(q^2+q+1)$. 

Now a tangent \orsp\ that meets $\li$ in the point $T$ contains $q^2$ \orsl
s that are not through $T$. By \cite[Theorem 2.5]{barw12}, each of these
sublines corresponds to a special normal
rational curve in some 3-space about a spread element. Moreover, this
correspondence is exact. Hence a ruled surface $\bbb$ that
corresponds to a tangent  \orsp\ has a exactly one conic directrix, and
$q^2$ normal rational curve directrices. Thus the number of triples is
$q^2\times q^7(q^3-1)(q^2+q+1)$. This is the same as the number of triples
in Lemma~\ref{count-triples}. 

Hence the number of ruled 
surfaces satisfying the conditions of Theorem~\ref{hope1} is equal to the number of ruled surfaces that correspond to
tangent \orsp s. Hence every ruled surface satisfying the conditions of
Theorem~\ref{hope1} does indeed correspond
to a tangent \orsp\ as required.
\hfill $\square$

\end{document}